\documentclass[12pt]{article} 
\usepackage{amsfonts,amsmath,latexsym,amssymb,mathrsfs,amsthm}

\def\numberlikeadb{\global\def\theequation{\thesection.\arabic{equation}}}
\numberlikeadb
\newtheorem{theorem}{Theorem}[section]
\newtheorem{lemma}[theorem]{Lemma}
\newtheorem{corollary}[theorem]{Corollary}

\newtheorem{proposition}[theorem]{Proposition}
\newtheorem{remark}[theorem]{Remark}

\numberwithin{equation}{section}

\begin{document} 

\title{A Stein characterisation of the generalized hyperbolic distribution}
\author{Robert E. Gaunt\footnote{School of Mathematics, The University of Manchester, Manchester M13 9PL, UK.
}\\
The University of Manchester and University of Oxford
}

\date{March 2017} 
\maketitle

\begin{abstract}The generalized hyperbolic (GH) distributions form a five parameter family of probability distributions that includes many standard distributions as special or limiting cases, such as the generalized inverse Gaussian distribution, Student's $t$-distribution and the variance-gamma distribution, and thus the normal, gamma and Laplace distributions.  In this paper, we consider the GH distribution in the context of Stein's method.  In particular, we obtain a Stein characterisation of the GH distribution that leads to a Stein equation for the GH distribution.  This Stein equation reduces to the Stein equations from the current literature for the aforementioned distributions that arise as limiting cases of the GH superclass.
\end{abstract}


\noindent{{{\bf{AMS 2010 Subject Classification:}}} 60F05; 60E05

\section{Introduction}

In 1972, Stein \cite{stein} introduced a powerful technique for deriving distributional bounds for normal approximation.  Stein's method for normal approximation rests on the following characterization of the normal distribution: $W\sim N(0,\sigma^2)$ if and only if 
\begin{equation} \label{stein lemma}\mathbb{E}[\sigma^2f'(W)-Wf(W)]=0
\end{equation}
for all real-valued absolutely continuous functions $f$ such that $\mathbb{E}|f'(Z)|<\infty$ for $Z\sim N(0,\sigma^2)$.  This characterisation leads to the so-called Stein equation: 
\begin{equation} \label{normal equation} \sigma^2f'(x)-xf(x)=h(x)-N h,
\end{equation} 
where $N h$ denotes $\mathbb{E}h(Z)$ for $Z\sim N(0,\sigma^2)$, and the test function $h$ is real-valued.  Evaluating both sides of (\ref{normal equation}) at a random variable $W$ and taking expectations gives
\begin{equation} \label{expect} \mathbb{E}[\sigma^2f'(W)-Wf(W)]=\mathbb{E}h(W)-N h.
\end{equation}
We can thus bound the quantity $|\mathbb{E}h(W)-N h|$ by solving the Stein equation (\ref{normal equation}) and then bounding the left-hand side of (\ref{expect}).  For a detailed account of the method see the monograph \cite{chen} or the review article \cite{ross}.

Over the years, the method has been adapted to many other probability distributions, such as the Poisson \cite{chen 0}, gamma \cite{gaunt chi square, luk}, exponential \cite{chatterjee, pekoz1} and Laplace distribution \cite{dobler, pike}, and has been applied to a wide range of applications, including random matrix theory \cite{mackey}, random graph theory \cite{bhj92}, urn models \cite{dobler beta, goldstein4, pekoz3}, goodness-of-fit statistics \cite{gaunt friedman, gaunt chi square} and statistical physics \cite{cs11, em14}.  For an overview of the current literature see \cite{ley}.  In particular, the method has recently been extended to the variance-gamma distribution; see \cite{gaunt vg} and \cite{eichelsbacher}.  The variance-gamma distributions form a four parameter family of distributions that include as special or limiting cases the normal, gamma and Laplace distributions, as well the product of two (possibly correlated) mean zero normal random variables and the difference of two (possibly correlated) gamma random variables; for a list of the distributions contained in the variance-gamma class see \cite{gaunt vg}.  The variance-gamma distribution is a limiting case of the five parameter generalized hyperbolic (GH) distribution (see \cite{eberlein}), and it is therefore an intriguing prospect to extend Stein's method to this distribution.  This is the focus of this paper.

The GH distribution was introduced by Barndorff-Nielsen \cite{barndorff}, who studied it in connection with the modelling of dune movements.  For certain parameter values the GH distributions have semi-heavy tails which makes them appropriate for financial modelling; see, for example, \cite{bibby, eberlein1, eberlein2, madan}.  The distribution has five parameters $\lambda \in\mathbb{R}$, $\delta >0$, $\mu \in \mathbb{R}$, $\alpha >|\beta| \geq 0$ and its probability density function is 
\begin{align}  p_{GH}(x;\lambda,\alpha,\beta,\delta,\mu) &= \frac{(\alpha^2 - \beta^2)^{\lambda/2}}{\sqrt{2\pi} \alpha^{\lambda -\frac{1}{2}}\delta^{\lambda}K_{\lambda}(\delta\sqrt{\alpha^2-\beta^2})}\mathrm{e}^{\beta (x-\mu )}(\delta^2+(x-\mu)^2)^{(\lambda-\frac{1}{2})/2} \nonumber \\ 
\label{sevenz} &\quad\times K_{\lambda-\frac{1}{2}}(\alpha\sqrt{\delta^2+(x-\mu)^2}), \quad x\in\mathbb{R},
\end{align}
where $K_\nu$ is the modified Bessel function of the second kind (see the Appendix for a definition).  If a random variable $X$ has density (\ref{sevenz}), we write $X\sim GH(\lambda,\alpha,\beta,\delta,\mu)$.  The parameter $\mu$ is the location parameter, and in this paper we shall often set it to 0 to simplify the exposition; results for the general case then follow by a simple linear transformation, because $GH(\lambda,\alpha,\beta,\delta,\mu)\stackrel{\mathcal{D}}{=}\mu+GH(\lambda,\alpha,\beta,\delta,0)$.  The support of the distribution is $\mathbb{R}$, except for certain limiting cases, in which case the support is either a positive or negative half-line. We collect these special cases in Section 2 (see \cite{eberlein} for a thorough treatment of all the limiting cases).  Indeed, the combination of the five parameters of the model and the flexibility of the modified Bessel function $K_\nu$ allow for a wide variety of distributions to fall into the GH family, including the generalized inverse Gaussian (GIG) distribution, Student's $t$-distribution and also the variance-gamma distribution.  We present a number of important special and limiting cases in Section 2.

In extending Stein's method to a new distribution, the first step is typically to find a suitable characterising equation for the distribution (for the normal distribution this is (\ref{stein lemma})), from which we obtain a corresponding Stein equation.  In Proposition \ref{ghds}, we obtain a characterising equation for the $GH(\lambda,\alpha,\beta,\delta,0)$ distribution, which leads to the following Stein equation for the $GH(\lambda,\alpha,\beta,\delta,0)$ distribution:
\begin{align}\mathcal{A}_{\lambda,\alpha,\beta,\delta}f(x)&:=\frac{x^2+\delta^2}{x}f''(x)+\bigg[2\lambda+2\beta x+\frac{2\beta\delta^2}{x}-\frac{\delta^2}{x^2}\bigg]f'(x)\nonumber\\
\label{steinsing}&\quad+\bigg[2\lambda\beta-(\alpha^2-\beta^2)x+\frac{\beta^2\delta^2}{x}-\frac{\beta\delta^2}{x^2}\bigg]f(x)
=h(x)-\mathrm{GH}_{\delta,0}^{\lambda,\alpha,\beta}h,
\end{align}
where $\mathrm{GH}_{\delta,\mu}^{\lambda,\alpha,\beta}h$ denotes $\mathbb{E}h(X)$ for $X\sim GH(\lambda,\alpha,\beta,\delta,\mu)$.  We shall refer to $\mathcal{A}_{\lambda,\alpha,\beta,\delta}$ as the $GH(\lambda,\alpha,\beta,\delta,0)$ Stein operator, in accordance with standard practice from the literature.  This Stein equation has a number of interesting properties.  In particular, in Section 3.2, we see that, for particular parameter values, (\ref{steinsing}) reduces to the known Stein equations for the GIG \cite{koudou}, Student's $t$ \cite{schoutens} and variance-gamma \cite{gaunt vg} distributions.  The only other instance in the literature in which a comparably large class of distributions can be treated in one framework are the Stein equations for Pearson and Ord class of distributions (see \cite{schoutens}).  

The Stein equation (\ref{steinsing}) is a second order differential equation in $f$, $f'$ and $f''$.  Such Stein equations were uncommon in the literature, although recently \cite{gaunt vg, pekoz, pike} have obtained second order Stein equations.  In fact, $n$-th order Stein equations have recently been obtained for the product of $n$ independent beta, gamma and central normal random variables \cite{gaunt pn, gaunt ngb} and general linear combinations of gamma random variables \cite{aaps16}.  

The fact that (\ref{steinsing}) is a second order differential equation in $f$, $f'$ and $f''$ means that the standard generator \cite{barbour2, gotze} and density \cite{stein2} methods (the scope of the density method has recently been extended by \cite{ley}) do not easily lend themselves to deriving this Stein equation. (Note that a direct application of the density method would lead to a first order Stein equation with complicated coefficients involving the modified Bessel function of the second kind, which would most likely be intractable in applications.)  To arrive at our characterisation of the GH distribution that leads to (\ref{steinsing}), we note a second order homogeneous differential equation that the density (\ref{sevenz}) satisfies and then exploit the duality between Stein equations and differential equations of densities to obtain our Stein equation.  This is one the first explicit examples in the literature of this technique being used to derive a Stein equation, but see \cite{ley}, Section 3.6 for a discussion and example of this approach.  

Despite these interesting properties, the presence of singularities at $x=0$ in the Stein equation (\ref{steinsing}) suggests that it may not be easily tractable for proving approximation theorems via Stein's method.  We can remove the singularities by making the substitution $f(x)=x^2g(x)$, which leads to the following alternative Stein equation for the $GH(\lambda,\alpha,\beta,\delta,0)$ distribution:
\begin{align}\label{solnst}&x(x^2+\delta^2)g''(x)+\big(3\delta^2+2\beta\delta^2x+(2\lambda+4)x^2+2\beta x^3\big)g'(x)\nonumber \\
&+\big(3\beta\delta^2+(4\lambda+\beta^2\delta^2+2)x+(2\lambda+4)\beta x^2-(\alpha^2-\beta^2)x^3\big)g(x)=h(x)-\mathrm{GH}_{\delta,0}^{\lambda,\alpha,\beta}h.
\end{align}
At first glance, (\ref{solnst}) seems to be more tractable than (\ref{steinsing}).  However, there is also a problem with working with this Stein equation in applications.  The solution $g$ to (\ref{solnst}) is equal to $f(x)/x^2$, where $f$ is the solution to (\ref{steinsing}), and is given by formula (\ref{ink}) in Proposition \ref{solnlemma}.  It is clear from (\ref{ink}) that $g(x)=f(x)/x^2$  has a singularity at $x=0$ (except possibly for very particular choices of test function $h$).  As was shown by \cite{gaunt normal}, it is sometimes possible to prove approximation theorems by Stein's method when the derivatives of the solution of a Stein equation are unbounded as $|x|\rightarrow\infty$.  Although, the prospect of dealing with a singularity at an interior point seems more challenging and will not be addressed in this paper.  For dealing with the full class, the Stein equations (\ref{steinsing}) and (\ref{solnst}) may therefore not be tractable, and alternative methods may be needed to circumvent the Stein equation, as was done in the recent works \cite{aaps16, gaunt normal}.  It should, however, be noted that for particular parameter values, the Stein equation is highly tractable.  For example, the case $\delta\rightarrow0$ yields the variance-gamma class for which approximation theorems have been obtained in \cite{eichelsbacher, gaunt vg}.  

The rest of the article is organised as follows.  In Section 2, we present some basic properties of the GH distributions and note some important special and limiting cases of this superclass of distributions.  In Section 3, we obtain a characterising equation for the GH distribution that leads to the Stein equation (\ref{steinsing}).  In Lemma \ref{solnprop}, we solve the Stein equation, and in Lemma \ref{solnlemma} we prove that the solution and its first derivative are bounded when the test function $h$ is bounded.  In Section 3.2, we note that the Stein equation (\ref{steinsing}) reduces to the known GIG, Student's $t$ and variance-gamma Stein equations for appropriate parameter values.  The Appendix contains some elementary properties of modified Bessel functions that are used throughout this paper.

\section{The class of generalized hyperbolic distributions}

In this section, we present some basic properties of GH distributions, some of which we shall need in Section 3.  The GH distributions are studied in some detail in \cite{eberlein}, and we refer the reader to this work for further properties.  The class of GH distributions can be obtained by mean-variance mixtures of normal distributions where the mixing distribution is the GIG distribution \cite{bks82}.  Let $GIG(\lambda,a,b)$ denote the GIG distribution with parameters $\lambda\in\mathbb{R}$, $a>0$, $b>0$ and density
\begin{equation}\label{gigpdf}p_{GIG}(x;\lambda,a,b)=\bigg(\frac{a}{b}\bigg)^{\lambda/2}\frac{1}{2K_\lambda(\sqrt{ab})}x^{\lambda-1}\mathrm{e}^{-(ax+bx^{-1})/2}, \quad x>0.
\end{equation}
If $X\sim N(0,1)$ and $V\sim GIG(\lambda,\sqrt{\alpha^2-\beta^2},\delta)$ are independent, then 
\begin{equation*}\mu+\beta V+\sqrt{V}X\sim GH(\lambda,\alpha,\beta,\delta,\mu).
\end{equation*}
The moment generating function of the $GH(\lambda,\alpha,\beta,\delta,\mu)$ distribution is given by
\begin{equation}\label{ghdmgf}M(t)=\frac{\mathrm{e}^{\mu t}\gamma^\lambda}{(\alpha^2-(\beta+t)^2)^{\lambda/2}}\frac{K_\lambda(\delta\sqrt{\alpha^2-(\beta+t)^2})}{K_\lambda(\delta\gamma)},
\end{equation}
for all $t\in\mathbb{R}$ such that $|\beta+t|<\alpha$.  The mean and variance of $X\sim GH(\lambda,\alpha,\beta,\delta,\mu)$ are given by
\begin{align}\label{mean}\mathbb{E}X&=\mu+\frac{\delta\beta K_{\lambda+1}(\delta\gamma)}{\gamma K_\lambda(\delta\gamma)},\\
\label{variance} \mathrm{Var}X&=\frac{\delta K_{\lambda+1}(\delta\gamma)}{\gamma K_\lambda(\delta\gamma)}+\frac{\beta^2\delta^2}{\gamma^2}\bigg(\frac{K_{\lambda+2}(\delta\gamma)}{K_\lambda(\delta\gamma)}-\frac{K_{\lambda+1}^2(\delta\gamma)}{K_\lambda^2(\delta\gamma)}\bigg),
\end{align}
where $\gamma=\sqrt{\alpha^2-\beta^2}$.  In fact, GH distributions possess moments of arbitrary order (except in a limiting case that corresponds to Student's $t$-distribution), and formulas for these moments are given in \cite{scott}.  The class is closed under affine transformations \cite{eberlein}: if $X\sim GH(\lambda,\alpha,\beta,\delta,\mu)$ then $aX+b\sim GH(\lambda,\alpha/|a|,\beta/a,\delta|a|,a\mu+b)$.  However, the class is not closed under convolutions, meaning that the sum of two independent GH distributed random variables is no longer GH distributed.  The exceptions being the normal inverse Gaussian distribution ($\lambda=-1/2$) \cite{eberlein} and the variance-gamma distribution ($\delta\rightarrow0$) \cite{bibby}.

Except for particular limiting cases (which we note below), the density of the GH distribution is bounded across its support.  As is seen in the following lemma (which the author could not find in the literature), the tails of the GH distribution are semi-heavy, which makes the class appropriate for financial modelling.  We shall make use of this lemma in the proof of Proposition \ref{ghds}.

\begin{lemma}\label{lem0}The density $p$ of the $GH(\lambda,\alpha,\beta,\delta,\mu)$ has the following tail behaviour:
\begin{equation}\label{pasy}p(x)\sim \frac{(\alpha^2-\beta^2)^{\lambda/2}}{2(\alpha\delta)^\lambda K_\lambda(\delta\sqrt{\alpha^2-\beta^2})}|x|^{\lambda-1}\mathrm{e}^{-\alpha|x-\mu|+\beta (x-\mu)}\sum_{k=0}^\infty\frac{a_k}{x^k}, \quad |x|\rightarrow\infty,
\end{equation}
where the $a_k$ are constants independent of $x$, and $a_0=1$.  Therefore, $p(x)=O(|x|^{\lambda-1}\mathrm{e}^{-\alpha|x|+\beta x})$ and $p'(x)=O(|x|^{\lambda-1}\mathrm{e}^{-\alpha|x|+\beta x})$, as $|x|\rightarrow\infty$.  
\end{lemma}

\begin{proof}Applying the asymptotic formula (\ref{Ktendinfinity}) to (\ref{sevenz}), we have that
\begin{align*}p(x)&\sim \frac{(\alpha^2 - \beta^2)^{\lambda/2}}{\sqrt{2\pi} \alpha^{\lambda -\frac{1}{2}}\delta^{\lambda}K_{\lambda}(\delta\sqrt{\alpha^2-\beta^2})}\mathrm{e}^{\beta (x-\mu )}(\delta^2+(x-\mu)^2)^{(\lambda-\frac{1}{2})/2} \\
&\quad\times\sqrt{\frac{\pi}{2\alpha\sqrt{\delta^2+(x-\mu)^2}}}\mathrm{e}^{-\alpha\sqrt{\delta^2+(x-\mu)^2}}\sum_{k=0}^\infty\frac{a_k(\lambda-1/2)}{x^k}, \quad |x|\rightarrow\infty,
\end{align*}
where the $a_k(\lambda-1/2)$ are given by (\ref{aknu}).  Using the asymptotic formula
\begin{equation*}\sqrt{\delta^2+(x-\mu)^2}=|x-\mu|\sqrt{1+\delta^2/(x-\mu)^2}\sim |x-\mu|+O(x^{-1}), \quad |x|\rightarrow\infty,
\end{equation*}
now yields (\ref{pasy}).  Note that in general the $a_k$ differ from the $a_k(\lambda-1/2)$ except for $a_0=a_0(\lambda-1/2)=1$, the constant for the leading term in the asymptotic series.  The final assertion regarding the order of $p(x)$ and $p'(x)$ as $|x|\rightarrow\infty$ is immediate from the asymptotic series (\ref{pasy}).
\end{proof} 

The special and limiting cases of the GH distribution were studied in some detail in \cite{eberlein}.  We now record the relevant (non-degenerate) cases.  For derivations of these cases see \cite{eberlein}, although note that they all can be derived directly from the density (\ref{sevenz}) using the properties (\ref{spheress}) -- (\ref{Ktendinfinity}) of the modified Bessel function of the second kind.

\vspace{2mm}

\noindent{\bf{Hyperbolic distribution.}}  This distribution corresponds to the case $\lambda=1$:  
\begin{equation*}p_H(x;\alpha,\beta,\delta,\mu)=\frac{\sqrt{\alpha^2-\beta^2}}{2\alpha\delta K_1(\delta\sqrt{\alpha^2-\beta^2})}\mathrm{e}^{-\alpha\sqrt{\delta^2+(x-\mu)^2}+\beta(x-\mu)}, \quad x\in\mathbb{R}.
\end{equation*}

\vspace{2mm}

\noindent{\bf{Normal-inverse Gaussian distribution.}}  Taking $\lambda=-1/2$ gives 
\begin{equation*}p_{NIG}(x;\alpha,\beta,\delta,\mu)=\frac{\alpha\delta}{\pi\sqrt{\delta^2+(x-\mu)^2}}\mathrm{e}^{\delta\gamma+\beta(x-\mu)}K_1(\alpha\sqrt{\delta^2+(x-\mu)^2}), \quad x\in\mathbb{R}.
\end{equation*}

\vspace{2mm}

\noindent{\bf{GIG distribution}} with parameters $\lambda\in\mathbb{R}$, $a>0$, $b>0$.  Taking $\beta=\alpha-\frac{a}{2}$, $\alpha\rightarrow\infty$, $\delta\rightarrow0$, $\alpha\delta^2\rightarrow b$, $\mu=0$ in the density (\ref{sevenz}) of the $GH(\lambda,\alpha,\beta,\delta,\mu)$ distribution yields the $GIG(\lambda,a,b)$ distribution with density (\ref{gigpdf}).  Note that the support of this distribution is the positive real line.  Also, as noted by \cite{eberlein}, the distribution  $-GIG(\lambda,a,b)$, which has support on the negative real line is in the GH class.  Taking the limit $b\rightarrow 0$ yields the gamma distribution as a limiting case, and the limit $a\rightarrow0$ corresponds to inverse-gamma distribution.  For the gamma case ($b\rightarrow 0$) the distribution has singularity at 0 when $0<\lambda<1$; otherwise, the GIG distribution is bounded across its support.

\vspace{2mm}

\noindent{\bf{Student's $t$-distribution.}}  Letting $\lambda=-\frac{\nu}{2}<0$ and taking the limit $\alpha,\beta\rightarrow0$ gives a scaled and shifted Student's $t$-distribution with $\nu$ degrees of freedom: 
\begin{equation} \label{fancy} p_t(x;\nu,\delta,\mu)=\frac{\Gamma(\frac{\nu+1}{2})}{\sqrt{\pi \delta^2}\Gamma(\frac{\nu}{2})}\bigg(1+\frac{(x-\mu)^2}{\delta^2}\bigg)^{-\frac{1}{2}(\nu+1)}, \quad x\in\mathbb{R}.
\end{equation}
In the case $\delta =\sqrt{\nu}$ the density (\ref{fancy}) is that of Student's $t$-distribution with $\nu$ degrees of freedom, and specialising further to $\nu=1$ yields the Cauchy distribution.

\vspace{2mm}

\noindent{\bf{Variance-gamma distribution.}}  Letting $\nu=\lambda-1/2>-1/2$ and taking the limit $\delta\rightarrow0$ gives the $\mathrm{VG}_2(\nu,\alpha,\beta,\mu)$ distribution (this notation is consistent with that of \cite{gaunt vg}) with density
\begin{equation*}  p_{\mathrm{VG}_2}(x;\nu,\alpha,\beta,\mu) = \frac{(\alpha^2 - \beta^2)^{\nu + 1/2}}{\sqrt{\pi} \Gamma(\nu + \frac{1}{2})}\left(\frac{|x-\mu|}{2\alpha}\right)^{\nu} \mathrm{e}^{\beta (x-\mu)}K_{\nu}(\alpha|x-\mu|), \quad x\in \mathbb{R}.
\end{equation*}
The variance-gamma distribution does itself possess a number of standard probability distributions as special and limiting cases; see \cite{gaunt vg}.  To identify these cases, it is more helpful to work with an alternative parametrisation of the distribution.  Letting  
\begin{equation}\label{vgpar}\nu=\frac{r-1}{2}, \quad \alpha =\frac{\sqrt{\theta^2 +  \sigma^2}}{\sigma^2},\quad \beta =\frac{\theta}{\sigma^2}
\end{equation} 
gives the $\mathrm{VG}_1(r,\theta,\sigma,\mu)$ distribution, which has density
\begin{equation*} p_{\mathrm{VG}_1}(x;r,\theta,\sigma,\mu) = \frac{1}{\sigma\sqrt{\pi} \Gamma(\frac{r}{2})} \mathrm{e}^{\frac{\theta}{\sigma^2} (x-\mu)} \bigg(\frac{|x-\mu|}{2\sqrt{\theta^2 +  \sigma^2}}\bigg)^{\frac{r-1}{2}} K_{\frac{r-1}{2}}\bigg(\frac{\sqrt{\theta^2 + \sigma^2}}{\sigma^2} |x-\mu| \bigg). 
\end{equation*}
The support of the $\mathrm{VG}_1(r,\theta,\sigma,\mu)$ distribution is $\mathbb{R}$ when $\sigma>0$, but in the limit $\sigma\rightarrow 0$ the support is the region $(\mu,\infty)$ if $\theta>0$, and is $(-\infty,\mu)$ if $\theta<0$.  If $\sigma>0,$ the distribution is bounded across its support, except for when $r\leq1$ for which there is a singularity at the point $\mu$ of order $-\log|x-\mu|$ if $r=1$ and of order $|x-\mu|^{r-1}$ if $0<r<1$ (see \cite{gaunt vg}, p$.$ 5).  In \cite{gaunt vg}, Proposition 1.3 it is shown that
\begin{align*}N(\mu,\sigma^2)&\stackrel{\mathcal{D}}{=}\lim_{r\rightarrow\infty}\mathrm{VG}(r,0,\sigma/\sqrt{r},\mu), \\
\Gamma(r,\lambda)&\stackrel{\mathcal{D}}{=}\lim_{\sigma\downarrow0}\mathrm{VG}(2r,(2\lambda)^{-1},\sigma,0), \\
\mathrm{Laplace}(\mu,\sigma)&\stackrel{\mathcal{D}}{=}\mathrm{VG}(2,0,\sigma,\mu),
\end{align*}
and a number of other interesting cases are included, such as the product of two (possibly correlated) mean zero normal random variables and the difference of two (possibly correlated) gamma random variables.  Here, $\mathrm{Laplace}(\mu,\sigma)$ denotes a Laplace distribution with density $\frac{1}{2\sigma}\exp\big(-\frac{|x-\mu|}{\sigma}\big)$, $x\in\mathbb{R}$.  

\section{A Stein equation for the generalized hyperbolic distribution}

\subsection{A Stein characterisation of the generalized hyperbolic distribution}

The following proposition gives a characterisation of the $GH(\lambda,\alpha,\beta,\delta,0)$ distribution which leads to the GH Stein equation (\ref{steinsing}).  We begin by obtaining a differential equation satisfied by the density of the $GH(\lambda,\alpha,\beta,\delta,0)$ distribution and then exploit the duality of Stein equations and differential equations satisfied by densities to prove the proposition.  To simplify the exposition, we take $\mu=0$ throughout this section.  We can deduce results for general case easily, because $GH(\lambda,\alpha,\beta,\delta,\mu)\stackrel{\mathcal{D}}{=}\mu+GH(\lambda,\alpha,\beta,\delta,0)$.

\begin{proposition}\label{ghds}Let $W$ be a real-valued random variable.  Then $W$ follows the  $GH(\lambda,\alpha,\beta,\delta,0)$ distribution if and only if
\begin{align}\label{ezero}\mathbb{E}[\mathcal{A}_{\lambda,\alpha,\beta,\delta}f(W)]=0
\end{align}
for all piecewise twice continuously differentiable function $f:\mathbb{R}\rightarrow\mathbb{R}$ for which $\mathbb{E}|\mathcal{A}_{\lambda,\alpha,\beta,\delta}f(Z)|<\infty$ for $Z\sim GH(\lambda,\alpha,\beta,\delta,0)$, and
\begin{equation}\label{fcond}\lim_{|x|\rightarrow\infty}|x|^\lambda \mathrm{e}^{-\alpha|x|+\beta x}f^{(k)}(x)=0
\end{equation}
for $k=0,1,2$, where $f^{(0)}\equiv f$.
\end{proposition}

Before proving Proposition \ref{ghds}, we establish several lemmas.

\begin{lemma}\label{lem1}Suppose $g$ and $h$ are twice differentiable.  Then the general solution to the differential equation
\begin{align}&w''(x)-\bigg(\frac{g''(x)}{g'(x)}-\frac{g'(x)}{g(x)}+\frac{2h'(x)}{h(x)}\bigg)w'(x)
-\bigg(\bigg(\frac{\nu^2}{g(x)^2}+1\bigg)g'(x)^2+\frac{h'(x)g'(x)}{h(x)g(x)}\nonumber\\
\label{de4}&\quad\quad+\frac{h(x)h''(x)-2h'(x)^2}{h(x)^2}-\frac{h'(x)g''(x)}{h(x)g'(x)}\bigg)w(x)=0
\end{align}
is given by $w(x)=Ah(x)K_\nu(g(x))+Bh(x)I_\nu(g(x))$, where $A$ and $B$ are arbitrary real-valued constants.  The modified Bessel functions $K_\nu$ and $I_\nu$ are defined in the Appendix.
\end{lemma}

\begin{proof}We begin by noting that the solution to the homogeneous ordinary differential equation
\begin{equation}\label{de1}r''(x)+\frac{1}{x}r'(x)-\bigg(1+\frac{\nu^2}{x^2}\bigg)r(x)=0
\end{equation}
is given by $r(x)=AK_{\nu} (x) +BI_{\nu} (x)$ (see (\ref{realfeel})).  The rest of the proof involves using (\ref{de1}) and simple manipulations to deduce that $w(x)=Ah(x)K_\nu(g(x))+Bh(x)I_\nu(g(x))$ solves (\ref{de4}).  We start by using (\ref{de1}) to obtain a differential equation satisfied by $s(x)=AK_\nu(g(x))+BI_\nu(g(x))$.  For $y=g(x)$, we have that $\frac{d}{dy}=\frac{1}{g'(x)}\frac{\mathrm{d}}{\mathrm{d}x}$ and $\frac{\mathrm{d}^2}{\mathrm{d}y^2}=\frac{1}{g'(x)^2}\frac{\mathrm{d}^2}{\mathrm{d}x^2}-\frac{g''(x)}{g'(x)^3}\frac{\mathrm{d}}{\mathrm{d}x}$.  Therefore $s$ satisfies
\begin{equation*}\frac{1}{g'(x)^2}s''(x)-\frac{g''(x)}{g'(x)^3}s'(x)+\frac{1}{g(x)g'(x)}s'(x)-\bigg(1+\frac{\nu^2}{g(x)^2}\bigg)s(x)=0,
\end{equation*}
which on multiplying through by $g'(x)^2$ becomes
\begin{equation}\label{de3}s''(x)-\bigg(\frac{g''(x)}{g'(x)}-\frac{g'(x)}{g(x)}\bigg)s'(x)-\bigg(1+\frac{\nu^2}{g(x)^2}\bigg)g'(x)^2s(x)=0.
\end{equation}
Substituting $w(x)=h(x)s(x)$ into (\ref{de3}), using the product rule for differentiation and simplifying then shows that $w$ does indeed solve (\ref{de4}), which completes the proof.
\end{proof}

\begin{corollary}\label{cor1}The density $p$ of the $GH(\lambda,\alpha,\beta,\delta,0)$ distribution, as given by (\ref{sevenz}), satisfies
\begin{align*}&\frac{x^2+\delta^2}{x}p''(x)+\bigg(-2(\lambda-1)-2\beta x-\frac{2\beta\delta^2}{x}-\frac{\delta^2}{x^2}\bigg)p'(x)\\
&\quad\quad\quad+\bigg(2(\lambda-1)\beta-(\alpha^2-\beta^2)x+\frac{\beta^2\delta^2}{x}+\frac{\beta\delta^2}{x^2}\bigg)p(x)=0.
\end{align*}
\end{corollary}

\begin{proof}Apply Lemma \ref{lem1} (taking $B=0$ and $A\not=0$) with $g(x)=\alpha\sqrt{\delta^2+x^2}$ and $h(x)=\mathrm{e}^{\beta x}(\delta^2+x^2)^{(\lambda-1/2)/2}$, and then multiply through by $\frac{x^2+\delta^2}{x}$.  We omit the tedious calculations.
\end{proof}

\begin{lemma}\label{solnprop}Let $h:\mathbb{R}\rightarrow\mathbb{R}$ be measurable. Then the function $f:\mathbb{R}\rightarrow\mathbb{R}$, as given by
\begin{align}f(x)&=-\frac{\mathrm{e}^{-\beta x}K_{\nu}(\alpha\sqrt{\delta^2+x^2})}{(\delta^2+x^2)^{\nu/2}}\int_0^x \mathrm{e}^{\beta t}(\delta^2+t^2)^{\nu/2}I_{\nu}(\alpha\sqrt{t^2+\delta^2})\tilde{h}(t)\,\mathrm{d}t \nonumber\\
\label{ink}&\quad-\frac{\mathrm{e}^{-\beta x}I_{\nu}(\alpha\sqrt{\delta^2+x^2})}{(\delta^2+x^2)^{\nu/2}}\int_x^{\infty} \mathrm{e}^{\beta t}(\delta^2+t^2)^{\nu/2}K_{\nu}(\alpha\sqrt{t^2+\delta^2})\tilde{h}(t)\,\mathrm{d}t \\
&=-\frac{\mathrm{e}^{-\beta x}K_{\nu}(\alpha\sqrt{\delta^2+x^2})}{(\delta^2+x^2)^{\nu/2}}\int_0^x \mathrm{e}^{\beta t}(\delta^2+t^2)^{\nu/2}I_{\nu}(\alpha\sqrt{t^2+\delta^2})\tilde{h}(t)\,\mathrm{d}t \nonumber\\
\label{pen}&\quad+\frac{\mathrm{e}^{-\beta x}I_{\nu}(\alpha\sqrt{\delta^2+x^2})}{(\delta^2+x^2)^{\nu/2}}\int_{-\infty}^x \mathrm{e}^{\beta t}(\delta^2+t^2)^{\nu/2}K_{\nu}(\alpha\sqrt{t^2+\delta^2})\tilde{h}(t)\,\mathrm{d}t,
\end{align}
solves the $GH(\lambda,\alpha,\beta,\delta,0)$ Stein equation (\ref{steinsing}), where $\tilde{h}(x)=h(x)-\mathrm{GH}_{\delta,0}^{\nu,\alpha,\beta}h$ and $\nu=\lambda-1/2$.  Notice that the signs and range of integration in the final terms of (\ref{ink}) and (\ref{pen}) differ.
\end{lemma}

\begin{proof}Since (\ref{steinsing}) is an inhomogeneous linear ordinary differential equation, we can use the method of variation of parameters (see \cite{collins} for an account of the method) to write down the general solution of (\ref{steinsing}):
\begin{equation}\label{vargensoln}f(x)=-w_1(x)\int_a^x\frac{tw_2(t)\tilde{h}(t)}{(\delta^2+t^2)W(t)}\,\mathrm{d}t+w_2(x)\int_b^x\frac{tw_1(t)\tilde{h}(t)}{(\delta^2+t^2)W(t)}\,\mathrm{d}t,
\end{equation}
where $w_1$ and $w_2$ are solutions to the homogeneous equation
\begin{align}\label{de5}&\frac{x^2+\delta^2}{x}w''(x)+\bigg[2\lambda+2\beta x+\frac{2\beta\delta^2}{x}-\frac{\delta^2}{x^2}\bigg]w'(x)\nonumber\\
&\quad+\bigg[2\lambda\beta-(\alpha^2-\beta^2)x+\frac{\beta^2\delta^2}{x}-\frac{\beta\delta^2}{x^2}\bigg]w(x)=0
\end{align}
and $a$ and $b$ are arbitrary constants, and $W(t)=W(w_1,w_2)=w_1w_2'-w_2w_1'$ is the Wronskian.

A simple application of Lemma \ref{lem1} gives that a pair of linearly independent solutions to (\ref{de5}) is given by
\begin{equation*}w_1(x)=\frac{\mathrm{e}^{-\beta x}K_{\nu}(\alpha\sqrt{\delta^2+x^2})}{(\delta^2+x^2)^{\nu/2}}, \quad w_2(x)=\frac{\mathrm{e}^{-\beta x}I_{\nu}(\alpha\sqrt{\delta^2+x^2})}{(\delta^2+x^2)^{\nu/2}}.
\end{equation*}
We now compute the Wronskian $W(w_1(x),w_2(x))$.  For differentiable functions $u$ and $v$, we have that
\begin{align*}W(u(x)K_\nu(v(x)),u(x)I_\nu(v(x)))&=u(x)^2v'(x)[K_\nu(v(x))I_\nu'(v(x))-K_\nu'(v(x))I_\nu(v(x))]\\
&=\frac{u(x)^2v'(x)}{v(x)},
\end{align*}
where we used the formula $W(K_\nu(x),I_\nu(x))=x^{-1}$ (\cite{olver}, formula 10.28.2) to obtain the final equality.  A simple calculation now gives that
\begin{equation*}W(w_1(x),w_2(x))=\frac{\mathrm{e}^{-2\beta x}}{x(\delta^2+x^2)^{\nu-1}}.
\end{equation*}
Substituting the relevant quantities into (\ref{vargensoln}) and taking $a=0$ and $b=\infty$ yields the solution (\ref{ink}).  That the solutions (\ref{ink}) and (\ref{pen}) are equal follows because $\mathrm{e}^{\beta t}(\delta^2+t^2)^{\nu/2}K_{\nu}(\alpha\sqrt{t^2+\delta^2})$ is proportional to the $GH(\lambda,\alpha,\beta,\delta,0)$ density function.
\end{proof}

\begin{lemma}\label{solnlemma}Suppose that $h:\mathbb{R}\rightarrow\mathbb{R}$ is bounded and let $f$ denote the solution (\ref{ink}).  Then there exist non-negative constants $C_0$ and $C_1$ such that $\|f\|\leq C_0\|\tilde{h}\|$ and $\|f'\|\leq C_1\|\tilde{h}\|$, where $\|\tilde{h}\|=\|\tilde{h}\|_\infty=\sup_{x\in\mathbb{R}}|\tilde{h}(x)|$.
\end{lemma}

\begin{proof}The equality between solutions (\ref{ink}) and (\ref{pen}) means that we can restrict our attention to the case $x\geq0$ provided that we deal with the cases of positive and negative $\beta$.  As $h$ is bounded, we have from (\ref{ink}) that, for $x\geq0$,
\begin{align*}
|f(x)|&\leq\|\tilde{h}\|\frac{\mathrm{e}^{-\beta x}K_{\nu}(\alpha\sqrt{\delta^2+x^2})}{(\delta^2+x^2)^{\nu/2}}\int_0^x \mathrm{e}^{\beta t}(\delta^2+t^2)^{\nu/2}I_{\nu}(\alpha\sqrt{t^2+\delta^2})\,\mathrm{d}t \\
&\quad+\|\tilde{h}\|\frac{\mathrm{e}^{-\beta x}I_{\nu}(\alpha\sqrt{\delta^2+x^2})}{(\delta^2+x^2)^{\nu/2}}\int_x^{\infty} \mathrm{e}^{\beta t}(\delta^2+t^2)^{\nu/2}K_{\nu}(\alpha\sqrt{t^2+\delta^2}))\,\mathrm{d}t.
\end{align*}
The solution $f$ is clearly bounded for all finite $x$.   Also, from the asymptotic formulas (\ref{Ktendinfinity}) and (\ref{roots}) for $K_\nu$ and $I_\nu$, it follows that $f$ is bounded as $x\rightarrow\infty$.  Let us now show that $f'$ is also bounded for all $x\in\mathbb{R}$.  Differentiating (\ref{ink}) and using that $h$ is bounded gives that, for $x\geq0$,
\begin{align*}|f'(x)|&\leq\|\tilde{h}\|\bigg|\frac{\mathrm{d}}{\mathrm{d}x}\bigg(\frac{\mathrm{e}^{-\beta x}K_{\nu}(\alpha\sqrt{\delta^2+x^2})}{(\delta^2+x^2)^{\nu/2}}\bigg)\bigg|\int_0^x \mathrm{e}^{\beta t}(\delta^2+t^2)^{\nu/2}I_{\nu}(\alpha\sqrt{t^2+\delta^2})\,\mathrm{d}t \\
&\quad+\|\tilde{h}\|\bigg|\frac{\mathrm{d}}{\mathrm{d}x}\bigg(\frac{\mathrm{e}^{-\beta x}I_{\nu}(\alpha\sqrt{\delta^2+x^2})}{(\delta^2+x^2)^{\nu/2}}\bigg)\bigg|\int_x^{\infty} \mathrm{e}^{\beta t}(\delta^2+t^2)^{\nu/2}K_{\nu}(\alpha\sqrt{t^2+\delta^2})\,\mathrm{d}t.
\end{align*}
Differentiating using the product rule, the formulas $\frac{\mathrm{d}}{\mathrm{d}x}\big(\frac{K_{\nu}(x)}{x^\nu}\big)=-\frac{K_{\nu+1}(x)}{x^\nu}$ and $\frac{\mathrm{d}}{\mathrm{d}x}\big(\frac{I_{\nu}(x)}{x^\nu}\big)=\frac{I_{\nu+1}(x)}{x^\nu}$ (\cite{olver}, formula 10.29.4) and the chain rule gives that
\begin{align*}\frac{\mathrm{d}}{\mathrm{d}x}\bigg(\frac{\mathrm{e}^{-\beta x}K_{\nu}(\alpha\sqrt{\delta^2+x^2})}{(\delta^2+x^2)^{\nu/2}}\bigg)&=-\mathrm{e}^{-\beta x}\bigg(\frac{K_{\nu}(\alpha\sqrt{\delta^2+x^2})}{(\delta^2+x^2)^{\nu/2}}+\alpha\frac{ K_{\nu+1}(\alpha\sqrt{\delta^2+x^2})}{(\delta^2+x^2)^{\nu/2+1/2}}\bigg), \\
\frac{\mathrm{d}}{\mathrm{d}x}\bigg(\frac{\mathrm{e}^{-\beta x}I_{\nu}(\alpha\sqrt{\delta^2+x^2})}{(\delta^2+x^2)^{\nu/2}}\bigg)&=-\mathrm{e}^{-\beta x}\bigg(\frac{I_{\nu}(\alpha\sqrt{\delta^2+x^2})}{(\delta^2+x^2)^{\nu/2}}-\alpha\frac{ I_{\nu+1}(\alpha\sqrt{\delta^2+x^2})}{(\delta^2+x^2)^{\nu/2+1/2}}\bigg).
\end{align*}
Therefore $f'$ is bounded for all finite $x$ and the asymptotic formulas (\ref{Ktendinfinity}) and (\ref{roots}) imply that it is bounded in the limit $x\rightarrow\infty$.  Since we dealt with the cases of both positive and negative $\beta$, it sufficed to prove boundedness for $x\geq0$, and the proof is thus complete. 
\end{proof}

\begin{remark}Looking back at the proof of Lemma \ref{solnprop}, due to the asymptotic behaviour of $I_\nu(x)$ as $x\rightarrow\infty$ (see (\ref{roots})), to ensure the solution is bounded as $x\rightarrow\infty$, we were forced to take $b=\infty$.  After doing so, we can take $a$ to be any finite real number and still obtain a bounded solution.  Therefore there are infinitely many bounded solutions to (\ref{steinsing}).  Choosing $a=0$ does however seem natural, and in fact on setting $\nu=\lambda-1/2$ and taking the limit $\delta\rightarrow0$ we recover the solution of the $\mathrm{VG}_1(\nu,\alpha,\beta,0)$ Stein equation that was given in Lemma 3.3 \cite{gaunt vg}.  That solution is in fact the unique bounded solution to $\mathrm{VG}_1(\nu,\alpha,\beta,0)$ Stein equation, provided $\nu\geq0$.
\end{remark}

\begin{remark}We shall use Lemma \ref{solnlemma} in our proof of Proposition \ref{ghds}.  In proving that proposition, we only need to make use of the fact that here exist non-negative constants $C_0$ and $C_1$ such that $\|f\|\leq C_0\|\tilde{h}\|$ and $\|f'\|\leq C_1\|\tilde{h}\|$; we do not need to find $C_0$ and $C_1$.  To determine explicit bounds for $f$, $f'$ and higher order derivatives of $f$, we would need to bound a number of expressions involving integrals of modified Bessel functions.  Bounds of expressions of a similar form to those we would encounter are given in \cite{gaunt uniform}, and were used to bound the derivatives of the solution of the variance-gamma Stein equation (see \cite{gaunt vg} and \cite{dgv15}).  However, since we do not use the GH Stein equation to prove any approximation results in this paper, we omit this analysis.  
\end{remark}

\noindent\emph{Proof of Proposition \ref{ghds}.} \emph{Necessity}.  Suppose that $W\sim GH(\lambda,\alpha,\beta,\delta,0)$.  For convenience, we let $A(x)=\frac{x^2+\delta^2}{x}$, $B(x)=2\lambda+\frac{2\beta\delta^2}{x}-\frac{\delta^2}{x^2}$ and $C(x)=2\lambda\beta-(\alpha^2-\beta^2)x+\frac{\beta^2\delta^2}{x}-\frac{\beta\delta^2}{x^2}$.  By two applications of integration by parts, we obtain
\begin{align}&\mathbb{E}[\mathcal{A}_{\lambda,\alpha,\beta,\delta}f(W)]=\mathbb{E}[A(W)f''(W)+B(W)f'(W)+C(W)f(W)]\nonumber\\
&=\int_{-\infty}^\infty \big\{A(x)f''(x)+B(x)f'(x)+C(x)f(x)\big\}p(x)\,\mathrm{d}x\nonumber \\
&=\int_{-\infty}^\infty\big\{A(x)p''(x)+(2A'(x)-B(x))p'(x)+(A''(x)-B'(x)+C(x))p(x)\big\}f(x)\,\mathrm{d}x\nonumber \\
\label{ibp1}&\quad+\Big[A(x)p(x)f'(x)\Big]_{-\infty}^\infty+\Big[\big\{B(x)p(x)-(A(x)p(x))'\big\}f(x)\Big]_{-\infty}^\infty.
\end{align}
Note the conditions imposed on $f$ ensure that the above expectation exists.  Straightforward calculations show that
\begin{eqnarray*}2A'(x)-B(x)&=&-2(\lambda-1)-2\beta x-\frac{2\beta\delta^2}{x}-\frac{\delta^2}{x^2}, \\
A''(x)-B'(x)+C(x)&=&2(\lambda-1)\beta-(\alpha^2-\beta^2)x+\frac{\beta^2\delta^2}{x}+\frac{\beta\delta^2}{x^2}.
\end{eqnarray*}
Therefore, by Corollary \ref{cor1}, the integral of equation (\ref{ibp1}) is equal to 0.  We also have that
\begin{eqnarray*}A(x)p(x)f'(x)&=&\frac{x^2+\delta^2}{x}p(x)f'(x),\\
\big\{B(x)p(x)-(A(x)p(x))'\big\}f(x)&=&\big\{B(x)p(x)-A'(x)p(x)-A(x)p'(x)\big\}f(x) \\
&=&\bigg[\bigg(2\lambda-1+\frac{2\beta\delta^2}{x}\bigg)p(x)-\frac{x^2+\delta^2}{x}p'(x)\bigg]f(x).
\end{eqnarray*}
The condition (\ref{fcond}) for $f$ and Lemma \ref{lem0} imply that $\lim_{|x|\rightarrow\infty}xp(x)f'(x)=\lim_{|x|\rightarrow\infty}p(x)f(x)=\lim_{|x|\rightarrow\infty}xp'(x)f(x)=0$, meaning that the final two terms of (\ref{ibp1}) are equal to 0.  This completes the proof of necessity.

\emph{Sufficiency}. Let $z\in\mathbb{R}$ and consider the $GH(\lambda,\alpha,\beta,\delta,0)$ Stein equation with test function $\mathbf{1}_{(0,z]}(x)$: 
\begin{align}\label{bcwhv}\mathcal{A}_{\lambda,\alpha,\beta,\delta}f(x)=\mathbf{1}_{(0,z]}(x)-F(z),
\end{align}
where $F(\cdot)$ is the cumulative distribution function of $Z\sim GH(\lambda,\alpha,\beta,\delta,0)$.  The solution to (\ref{bcwhv}) is given by (\ref{ink}) with $\tilde{h}(x)$ replaced by $\mathbf{1}_{(0,z]}(x)-F(z)$.  This solution is clearly piecewise twice continuously differentiable.  By Lemma \ref{solnlemma}, we have that the solution and its first derivative are bounded.  Also, rearranging the Stein equation (\ref{steinsing}) gives that
\begin{align}f_z''(x)&=\frac{x}{x^2+\delta^2}\bigg\{\bigg[2\lambda+2\beta x+\frac{2\beta\delta^2}{x}-\frac{\delta^2}{x^2}\bigg]f_z'(x) \nonumber\\
\label{feqn2}&\quad+\bigg[2\lambda\beta-(\alpha^2-\beta^2)x+\frac{\beta^2\delta^2}{x}-\frac{\beta\delta^2}{x^2}\bigg]f_z(x)+\mathbf{1}_{(0,z]}(x)-F(z)\bigg\}.
\end{align}
Since $f_z$, $f_z'$ and are bounded for all $x\in\mathbb{R}$, we have from (\ref{feqn2}) that $\lim_{|x|\rightarrow\infty}f_z''(x)=(\alpha^2-\beta^2)\lim_{|x|\rightarrow\infty}f_z(x)$, meaning that $f''(x)$ is bounded as $|x|\rightarrow\infty$.  Therefore $f_z$ satisfies condition (\ref{fcond}).  As $f_z$ satisfies (\ref{bcwhv}), it follows that $\mathbb{E}|\mathcal{A}_{\lambda,\alpha,\beta,\delta}f_z(Z)|<\infty$.  Hence, if (\ref{ezero}) holds for all piecewise twice continuously differentiable functions satisfying (\ref{fcond}) and such that $\mathbb{E}|\mathcal{A}_{\lambda,\alpha,\beta,\delta}f_z(Z)|<\infty$, then by (\ref{bcwhv}), we have that, for any $z\in\mathbb{R}$,
\begin{align*}0=\mathbb{E}[\mathcal{A}_{\lambda,\alpha,\beta,\delta}f(W)]=\mathbb{P}(W\leq z)-F(z),
\end{align*}
and so $\mathcal{L}(W)=GH(\lambda,\alpha,\beta,\delta,0)$. \hfill $\square$

\subsection{Limiting cases of the generalized hyperbolic Stein equation}

Here, we note a number of interesting limiting cases of the GH Stein equation (\ref{steinsing}).  In doing so, we note that many of the Stein equations from the current literature are limiting cases of (\ref{steinsing}).  For ease of presentation, we mostly present Stein operators, rather than Stein equations.

\vspace{2mm}

\noindent{\bf{Variance-gamma distribution.}}  Taking the limit $\delta\rightarrow0$ and setting $\nu=\lambda-1/2$ in the $GH(\lambda,\alpha,\beta,\delta,0)$ Stein operator (\ref{steinsing}) gives the following Stein operator for $\mathrm{VG}_2(\nu,\alpha,\beta,0)$ distribution:
\begin{equation} \label{muggers} \mathcal{A}f(x)=xf''(x) + (2\nu + 1 +2\beta x)f'(x) + ((2\nu+1)\beta - (\alpha^2 -\beta^2)x)f(x), 
\end{equation}
which we recognise as the $\mathrm{VG}_2(\nu,\alpha,\beta,0)$ Stein operator that was obtained  by \cite{gaunt vg}.  A Stein operator for the $\mathrm{VG}_1(r,\theta,\sigma,0)$ distribution how follows on changing variables using (\ref{vgpar}) and then multiplying through by $\sigma^2$:
\begin{equation} \label{nice}\mathcal{A}f(x)= \sigma^2 xf''(x) + (\sigma^2 r + 2\theta x)f'(x) +(r\theta -  x)f(x).
\end{equation}
A Stein equation for the $\mathrm{VG}_1(r,\theta,\sigma,0)$ distribution is therefore given by
\begin{equation*}\sigma^2 xf''(x) + (\sigma^2 r + 2\theta x)f'(x) +(r\theta -  x)f(x)=h(x)-\mathbb{E}h(X),
\end{equation*}
where $X\sim\mathrm{VG}_1(r,\theta,\sigma,0)$.  The solution of this Stein equation and its first derivative are bounded if $h$ is bounded (see \cite{gaunt vg}, Lemma 3.3), and it was shown by \cite{dgv15} that the $n$-th derivative of the solution is bounded if all derivatives of $h$ up to $(n-1)$-th order are bounded.

As was noted by \cite{gaunt vg}, there are number of interesting special and limiting cases of the Stein operator (\ref{nice}).  Letting $r=2s$, $\theta=(2\lambda)^{-1}$, $\mu=0$ and taking the limit $\sigma\rightarrow 0$ in (\ref{nice}) and then multipying through by $\lambda$ gives the Stein operator
\[\mathcal{A}f(x)=xf'(x) +(s - \lambda x)f(x), \]
which is the classical $\Gamma(s,\lambda)$ Stein operator (see \cite{diaconis} and \cite{luk}).  

We also note that a Stein operator for the $\mathrm{VG}_1(r,0,\sigma/\sqrt{r},0)$ distribution is
\[\mathcal{A}f(x)=\frac{\sigma^2}{r}xf''(x)+\sigma^2f'(x)-xf(x),\]
which in the limit $r\rightarrow\infty$ is the classical $N(0,\sigma^2)$ Stein operator. 

Since $\mathrm{Laplace}(0,\sigma)\stackrel{\mathcal{D}}{=}\mathrm{VG}_1(2,0,\sigma,0)$, we have the following Stein operator for the Laplace distribution:
\begin{equation*}\mathcal{A}f(x)=\sigma^2xf''(x)+2\sigma^2f'(x)-xf(x),
\end{equation*}
although note that this Stein operator is different from the Laplace Stein operator of \cite{pike}.  

Finally, taking $r=1$, $\theta=0$, $\sigma=\sigma_X\sigma_Y$ and $\mu=0$ in (\ref{nice}) gives the following Stein operator for distribution of the product of independent $N(0,\sigma_X^2)$ and $N(0,\sigma_Y^2)$ random variables (see part (iii) of Proposition 1.3 of \cite{gaunt vg}):
\begin{equation*}\mathcal{A}f(x)=\sigma_X^2\sigma_Y^2xf''(x) + \sigma_X^2\sigma_Y^2f'(x) -xf(x) ,
\end{equation*} 
which is in agreement with the Stein equation for the product of two independent, zero mean normal random variables that was obtained by \cite{gaunt pn}.

\vspace{2mm}

\noindent{\bf{Student's $t$ distribution.}}  Recall that the scaled and shifted Student's $t$-distribution with $\nu$ degrees of freedom, with density (\ref{fancy}) is a limiting case of the $GH(\lambda,\alpha,\beta,\delta,\mu)$ distribution with $\lambda=-\frac{\nu}{2}<0$ and $\alpha,\beta\rightarrow0$.  A Stein operator for this distribution (with $\mu=0$) can be obtained by taking these parameter values in the Stein operator (\ref{steinsing}):
\begin{equation*}\mathcal{A}f(x)=\frac{x^2+\delta^2}{x}f''(x)+\bigg(-\nu-\frac{\delta^2}{x^2}\bigg)f'(x).
\end{equation*}
Setting $f'(x)=xg(x)$ gives another Stein operator which has no singularity:
\begin{equation*}\mathcal{A}g(x)=(x^2+\delta^2)g'(x)-(\nu-1)xg(x).
\end{equation*}
Finally, applying a shift by $\mu$ gives a Stein operator the scaled and shifted Student's $t$-distribution with density (\ref{fancy}):
\begin{equation}\label{tstein}\mathcal{A}g(x)=((x-\mu)^2+\delta^2)g'(x)-(\nu-1)(x-\mu)g(x),
\end{equation}
which is in agreement with the Stein operator of \cite{dgv15} for the scaled and shifted Student's $t$-distribution.  In the special case $\delta=\sqrt{\nu}$, $\mu=0$, we obtain the Stein operator for Student's $t$-distribution that was obtained by \cite{schoutens}.  It is interesting to note that the solution of the Stein equation corresponding to (\ref{tstein}) is bounded if the test function $h$ is bounded, and that its $n$-th derivative is bounded if $\nu-2(n-1)>0$ and the derivatives of $h$ up to $(n-1)$-th order are bounded (see \cite{dgv15}).
 
\vspace{2mm}

\noindent{\bf{GIG distribution.}} Recall that the $GIG(\lambda,a,b)$ distribution arises as a limiting case of the $GH(\lambda,\alpha,\beta,\delta,\mu)$ distribution when we take $\beta=\alpha-\frac{a}{2}$, $\alpha\rightarrow\infty$, $\delta\rightarrow0$, $\alpha\delta^2\rightarrow b$ and $\mu=0$.  To obtain a Stein operator for the GIG distribution, we first consider the $GH(\lambda,\alpha,\beta,\delta,0)$ Stein operator applied to $f(x)=xg(x)$.  Using that $f'(x)=xg'(x)+g(x)$, $f''(x)=xg''(x)+2g'(x)$ and simplifying gives:
\begin{align*}\mathcal{A}g(x)&=(x^2+\delta^2)g''(x)+\bigg(2(\beta+1)x^2+2\lambda x+2(\beta+1)\delta^2-\frac{\delta^2}{x}\bigg)g'(x) \\
&\quad+\bigg(-(\alpha^2-\beta^2)x^2+2(\lambda+1)\beta x+2\lambda+\beta^2\delta^2+\frac{\beta\delta^2}{x}-\frac{\delta^2}{x^2}\bigg)g(x).
\end{align*}
In anticipation of letting $\alpha$ and $\beta$ tend to infinity, we divide through by $\beta$ to get another Stein operator:
\begin{align}\mathcal{A}g(x)&=\frac{x^2+\delta^2}{\beta}g''(x)+\bigg(2\bigg(1+\frac{1}{\beta}\bigg)x^2+\frac{2\lambda}{\beta} x+2\bigg(1+\frac{1}{\beta}\bigg)\delta^2-\frac{\delta^2}{\beta x}\bigg)g'(x)\nonumber \\
\label{giggg}&\quad+\bigg(-\frac{\alpha^2-\beta^2}{\beta}x^2+2(\lambda+1) x+\frac{2\lambda}{\beta}+\beta\delta^2+\frac{\delta^2}{x}-\frac{\delta^2}{\beta x^2}\bigg)g(x).
\end{align}
If $\beta=\alpha-\frac{a}{2}$, then $\alpha^2-\beta^2=a\beta+a^2/4$.  Now, taking $\beta=\alpha-\frac{a}{2}$, $\alpha\rightarrow\infty$, $\delta\rightarrow0$, $\alpha\delta^2\rightarrow b$ in (\ref{giggg}) gives the following Stein operator for the $GIG(\lambda,a,b)$ distribution:
\begin{equation}\label{steingi}\mathcal{A}g(x)=2x^2g'(x)+(-ax^2+2(\lambda+1)x+b)g(x),
\end{equation}
which we recognise as the $GIG(\lambda,a,b)$ Stein operator of \cite{koudou} (up to a multiple of 2).  The corresponding Stein equation is therefore
\begin{equation}\label{finl}2x^2g'(x)+(-ax^2+2(\lambda+1)x+b)g(x)=h(x)-\mathbb{E}h(X),
\end{equation}
where $X\sim GIG(\lambda,a,b)$.  To date, no bounds have been obtained for solution of the general $GIG(\lambda,a,b)$ Stein equation (\ref{finl}).  However, the limiting case $a\rightarrow0$, which is the inverse-gamma distribution, was treated by \cite{dgv15}.  They showed that the solution of the Stein equation (with $a=0$) is bounded if $\lambda<-1$ and the test function $h$ is bounded, and that the $n$-th derivative of the solution is bounded if $\lambda<-(2n+1)$ and the $n$-th derivative of $h$ exists and is bounded.

It actually turns out that solving the $GIG(\lambda,a,b)$ Stein equation (\ref{finl}) and then bounding the solution is quite straightforward.  This is because the density $p$, as given by (\ref{gigpdf}), satisfies the differential equation
\begin{equation*}(s(x)p(x))'=\tau(x)p(x),
\end{equation*}
where $s(x)=2x^2$ and $\tau(x)=-ax^2+2(\lambda+1)x+b$, and the Stein equation (\ref{finl}) can be written as $s(x)g'(x)+\tau(x)g(x)=h(x)-\mathbb{E}h(X)$.  Therefore, by Proposition 1  of \cite{schoutens}, we have that the Stein equation (\ref{finl}) has solution
\begin{align*}g(x)&=\frac{1}{2x^2p(x)}\int_0^x [h(t)-\mathbb{E}h(X)]p(t)\,\mathrm{d}t\\
&=-\frac{1}{2x^2p(x)}\int_x^\infty [h(t)-\mathbb{E}h(X)]p(t)\,\mathrm{d}t.
\end{align*}
Now, let $l=\mathbb{E}X=\frac{\sqrt{b}}{\sqrt{a}}\frac{K_{\lambda+1}(\sqrt{ab})}{K_\lambda(\sqrt{ab})}$ (see \cite{eberlein}).  Then, if $h$ is bounded, we can use Lemma 1 of \cite{schoutens0} to obtain the following bound for the solution:
\begin{equation*}\|g\|\leq\frac{1}{2l^2p(l)}\|h-\mathbb{E}h(X)\|.
\end{equation*}
Therefore, the solutions of the variance-gamma, Student's $t$ and GIG Stein equations that arise from the GH Stein equation (\ref{steinsing}) are all bounded provided the test function $h$ is bounded.

\subsection{Applications of Proposition \ref{ghds}}

The most common use of Stein characterisations of probability distributions is to identify an appropriate Stein equation for that distribution, which is then applied through Stein's method to prove approximation theorems.  However, as Stein characterisations characterise probability distributions, we can use them to infer various distributional properties, such as moments of distributions (see \cite{gaunt vg}), moment generating and characteristic functions (see \cite{gaunt pn} and \cite{gaunt vg}), and formulas for probability density functions (see \cite{gaunt ngb}, in which a new formula was established for the density of the distribution of the mixed product of independent beta, gamma and centred normal random variables).  We consider one such example here.

Suppose $W\sim GH(\lambda,\alpha,\beta,\delta,0)$.  Then taking $f(x)=x^k$ (note that $f$ satisfies the conditions of Proposition \ref{ghds}) and setting $M_k=\mathbb{E}W^k$ in the $GH(\lambda,\alpha,\beta,\delta,0)$  characterising equation (\ref{ezero}) leads to the following recurrence relation for the moments of the $GH(\lambda,\alpha,\beta,\delta,0)$ distribution:
\begin{align}(\alpha^2-\beta^2)M_{k+1}&=2\lambda\beta M_k+\big(k(k-1)+2\lambda k_\beta\delta^2\big)M_{k-1}+(2k-1)\beta\delta^2M_{k-2}\nonumber\\
\label{reqeqn}&\quad+k(k-2)\delta^2M_{k-3}.
\end{align}
We have that $M_0=1$, $M_1$ is given by (\ref{mean}), $M_2$ can be read off from (\ref{variance}), and a short calculation using the moment generating function (\ref{ghdmgf}) can be used to commute $M_3$.  Using forward substitution in (\ref{reqeqn}), one can then obtain formulas for moments of arbitrary order of the $GH(\lambda,\alpha,\beta,\delta,0)$ distribution.  The recurrence relation (\ref{reqeqn}) appears to be new, although it should be noted that \cite{scott} have already established a formula for the moments of general order of the GH distribution.

\appendix

\section{Elementary properties of modified Bessel functions}

Here we list standard properties of modified Bessel functions that are used throughout this paper.  All these formulas can be found in \cite{olver}.

The \emph{modified Bessel function of the first kind} of order $\nu \in \mathbb{R}$ is defined, for $x\in\mathbb{R}$, by
\begin{equation*}\label{defI}I_{\nu} (x) = \sum_{k=0}^{\infty} \frac{1}{\Gamma(\nu +k+1) k!} \left( \frac{x}{2} \right)^{\nu +2k}.
\end{equation*}
The \emph{modified Bessel function of the second kind} of order $\nu\in\mathbb{R}$ is defined, for $x>0$, by
\begin{equation*}K_\nu(x)=\int_0^\infty \mathrm{e}^{-x\cosh(t)}\cosh(\nu t)\,\mathrm{d}t.
\end{equation*}
The modified Bessel functions $K_\nu$ and $I_\nu$ form a pair of linearly independent solutions to the following differential equation:
\begin{equation} \label{realfeel} x^2 f''(x) + xf'(x) - (x^2 +\nu^2)f(x) =0.
\end{equation}
The modified Bessel function of the second kind is symmetric in the index parameter $\nu$:
\begin{eqnarray} \label{spheress} K_{\nu}(x)&=&K_{-\nu}(x), \quad \forall \nu\in\mathbb{R}, \\
 \label{sphere} K_{1/2}(x)&=&K_{-1/2}(x)=\sqrt{\frac{\pi}{2x}} \mathrm{e}^{-x}.
\end{eqnarray}
The modified Bessel functions have the following asymptotic behaviour.  Let $a_0(\nu)=1$ and
\begin{equation}\label{aknu}a_k(\nu)=\frac{(4\nu^2-1^2)(4\nu^2-3^2)\cdots(4\nu^2-(2k-1)^2)}{k!8^k}, \quad k=1,2,3,\ldots.
\end{equation}
Then
\begin{eqnarray}\label{Ktend0}K_{\nu} (x) &\sim& \begin{cases} 2^{|\nu| -1} \Gamma (|\nu|) x^{-|\nu|}, &  x \downarrow 0, \: \nu \not= 0, \\
-\log x, &  x \downarrow 0, \: \nu = 0, \end{cases} \\
\label{Ktendinfinity} K_{\nu} (x) &\sim& \sqrt{\frac{\pi}{2x}} \mathrm{e}^{-x}\sum_{k=0}^\infty\frac{a_k(\nu)}{x^k}, \quad x \rightarrow \infty, \\
\label{roots} I_{\nu} (x) &\sim& \frac{\mathrm{e}^x}{\sqrt{2\pi x}}\sum_{k=0}^\infty(-1)^k\frac{a_k(\nu)}{x^k}, \quad x \rightarrow \infty. 
\end{eqnarray}

\section*{Acknowledgements}The author acknowledges support from EPSRC
grant EP/K032402/1 and is currently supported by a Dame Kathleen Ollerenshaw Research Fellowship.  The author would like to thank the referees for their comments and is particularly grateful to one referee who spotted several errors in the displayed equations.

\end{document}